\documentclass[final]{siamltex}

\usepackage{times}

\usepackage{color}
\usepackage{graphicx}
\usepackage{amsmath,amssymb,nicefrac}
\usepackage{algorithm}
\usepackage{algpseudocode}
\usepackage{booktabs}
\usepackage{multirow}
\usepackage{longtable,lscape}

\newcommand{\be}{\begin{equation}}
\newcommand{\ee}{\end{equation}}

\newcommand{\mS}{\mathcal{S}}
\newcommand{\mN}{\mathcal{N}}
\newcommand{\id}{e}
\newcommand{\hA}{\widehat{A}}
\newcommand{\hb}{\widehat{b}}
\newcommand{\hx}{\widehat{x}}

\newcommand{\argminx}{\underset{x}{\operatorname{argmin\ }}}

\newcommand{\argminy}{\underset{y}{\operatorname{argmin\ }}}
\newcommand{\argminz}{\underset{z}{\operatorname{argmin\ }}}

\title{``Plug-and-Play'' Edge-Preserving Regularization\thanks{Received xxx}}

\author{Donghui Chen\thanks{School of Securities and Futures, Southwestern
        University of Finance and Economics ({\tt chendonghui@swufe.edu.cn}).}
        \and Misha E. Kilmer\thanks{Department of Mathematics, Tufts University
        {} ({\tt misha.kilmer@tufts.edu}).}
        \and Per Christian Hansen\thanks{Department of Applied Mathematics and
        Computer Science, Technical University of Denmark ({\tt pcha@ dtu.dk}).
        The author is supported by grant 274-07-0065 from the
        Danish Research Council for Technology and Production Sciences.} }

\begin{document}

\maketitle

\begin{abstract}
In many inverse problems it is essential to use regularization methods that
preserve edges in the reconstructions, and many reconstruction models have been
developed for this task, such as the Total Variation (TV) approach.
The associated algorithms are complex and require a good knowledge of large-scale
optimization algorithms, and they involve certain tolerances that the user must choose.
We present a simpler approach that relies only on standard computational building
blocks in matrix computations, such as orthogonal transformations, preconditioned
iterative solvers,
Kronecker products, and the discrete cosine transform\,---\,hence the term
``plug-and-play.''
We do not attempt to improve on TV reconstructions, but rather provide an easy-to-use
approach to computing reconstructions with similar properties.
\end{abstract}



\pagestyle{myheadings}
\thispagestyle{plain}
\markboth{D. CHEN, M. E. KILMER AND P. C. HANSEN}{``PLUG-AND-PLAY'' EDGE-PRESERVING REGULARIZATION}

\section{Introduction}\label{eppsec:intr}

This paper is concerned with  discretizations of linear ill-posed problems,
which arise in many technical and scientific applications such as astronomical
and medical imaging, geoscience, and non-destructive testing
\cite{hansenbook,MuellerSiltanen}.
The underlying model is $b = A\, \bar{x} + \eta$,
where $b$ is the noisy data, the matrix $A$ (which is often structured or sparse)
represents the forward operator, $\bar{x}$ is the exact solution, and $\eta$
denotes the unknown noise.
We present a new large-scale regularization algorithm which is able
to reproduce sharp gradients and edges in the solution.
Our algorithm uses only standard linear-algebra building blocks and is
therefore easy to implement and to tune to specific applications.

For ease of exposition, we focus on image deblurring problems involving
$m\times m$ images $B$ (the blurred and noisy image) and $X$ (the reconstruction).
With $b = \text{vec}(B)$ and $x = \text{vec}(X)$, both of length $n=m^2$,
the $n\times n$ matrix $A$ is determined by the point-spread function (PSF)
and corresponding boundary conditions~\cite{hanaol06}.
This matrix is very ill-conditioned (or rank deficient), and computing the ``naive solution''
$A^{-1}b = \bar{x}+A^{-1}\eta$ (or, in the rank-deficient case, the minimum norm solution)
results in a reconstruction that is completely dominated by the inverted noise~$A^{-1}\eta$.

Classical regularization methods, such as Tikhonov regularization or truncated SVD,
damp the noise component in the solution by suppressing high-frequency components
at the expense of smoothing the edges in the reconstruction.
The same is true for regularizing iterations (such as CGLS or GMRES) based on computing
solutions in a low-dimensional Krylov subspace.
The underlying characteristic in these methods is that regularization is achieved
by projecting the solution onto a low-dimensional \textit{signal subspace} $\mathcal{S}_k$
spanned by $k$, low-frequency basis vectors, with the result that the high-frequency components
are missing, hindering the reconstruction of sharp edges.

The projection approach is a powerful paradigm that can often be tailored to
the particular problem.
While these projected solutions may not always have satisfactory accuracy or details,
they still contain a large component of the desired solution, namely, the
low-frequency component which can be reliably determined from the noisy data.
What is missing is the high-frequency component, spanned by high-frequency basis vectors,
and this component must be determined via our prior information about the desired solution.

This work describes an easy-to-use large-scale method for computing the needed
high-frequency component,
related to the prior information that image must have smooth regions while the gradient
of the reconstructed image is allowed to have some (but not too many) large values.
This idea is similar in spirit to Total Variation regularization, where the gradient
is required to be sparse;
but by relaxing this constraint we arrive at problems that are simpler to solve.
The work can be considered as a continuation of earlier work \cite{PPTSVD2,pptsvd96,mtsvd92}
by one of us; it is also related to the decomposition approach in~\cite{BaglamaReichel}.

The remainder of this paper is organized as follows. In Section~\ref{eppsec:epp} we present
the new edge-preserving algorithm and the convergence analysis. Section~\ref{eppsec:alg} discusses the efficient numerical implementation issues. Section~\ref{eppsec:exp} presents numerical experiments of the new deblurring algorithm and comparisons with other state-of-art deblurring algorithms. The conclusions are presented in Section~\ref{eppsec:con}.

\section{The Projection-Based Edge-Preserving Algorithm}
\label{eppsec:epp}

This section pre\-sents the main ideas of the algorithm, while the implementation details
for large-scale problems are discussed in the next section.

\subsection{Mathematical Model}

Throughout, the matrix $L$ defines a discrete derivative or gradient of the solution
(to be precisely defined later), and $\| \cdot \|_p$ denotes the vector $p$-norm.
The underlying prior information is then that the solution's seminorm $\| L\, x \|_p$,
with $1< p<2$,
is not large (which allows some amount of large gradients or edges in the reconstruction).
The choice of the combination of $L$ and $p$ is important and, of course, somewhat problem dependent;
the matrix $L$ used here is the $2m(m-1)\times n$ matrix given by
  \[
    L = \begin{pmatrix} \!\!\!\! L_1 \otimes I \\ \ \ I \otimes L_1 \end{pmatrix}, \quad \mbox{where} \quad
    L_1 = \begin{pmatrix} -1 & 1 & 0& \cdots & 0 \\ 0 & -1 & 1& \cdots & 0 \\
          \vdots  & \vdots  & \ddots  & \ddots & \vdots \\ 0 & 0 & \cdots & -1& 1 \end{pmatrix}
    \in \mathbb{R}^{m-1\times m},
  \]
where $\otimes$ is the Kronecker product \cite{PPTSVD2}.
The one-dimensional null space $\mathcal{N}(L)$ of this matrix is spanned by
the $n$-vector $e$ of all ones.

Assume $W_k\in \mathbb{R}^{n\times k}$ is a matrix with orthonormal columns that span the
signal subspace $\mathcal{S}_k$, and let $W_0$ be the matrix containing the orthonormal
basis vectors for the complementary space $\mathcal{S}_k^\perp$.
The fundamental assumption is that the columns of $W_k$ represent ``smooth" modes
in which it is possible to distinguish a substantial component of the signal from the noise.
In other words, with the model from Section~1, we assume that
  \begin{equation}
  \label{eq:separate}
    \| W_k^T \bar{x} \|_2 \gg \| W_k^T (A^{-1} \eta) \|_2 \ .
  \end{equation}
This ensures that we can compute a good, but smooth, approximation to $\bar{x}$
as
  \[
    x_k = W_k\, y_k \ , \qquad y_k = \mathrm{argmin}_y \| (A\, W_k)\, y - b \|_2 \ ,
  \]
and we refer to the minimization problem for $y_k$ as the \textit{projected problem},
which we assume is easy to solve.
To obtain a reconstruction with the desired features,
our strategy is then to compute the solution of the following modified projection problem
  \be
  \label{eppeq:epp}
    \min_{x\in \mathcal{B}} \| L \, x \|_p \quad \text{with} \quad
    \mathcal{B} = \{ x: x = \argminz \| (A\,W_k W_k^T)z - b \|_2 \} \ ,
  \ee
with  $L$ defined above.
As we shall see, we can express the solution to (\ref{eppeq:epp}) as the low-frequency
solution $x_k$ plus a high-frequency correction.

\subsection{Uniqueness Analysis}

Eld\'en \cite{elden82} provides an explicit solution of~\eqref{eppeq:epp} for the case
$p=2$ and proves the uniqueness condition for the minimizer.
The MTSVD algorithm \cite{mtsvd92} corresponds the case where $p=2$ and $W_k$ consists of
the first $k$ right singular vectors, while the PP-TSVD algorithm \cite{pptsvd96}
and its 2D extension \cite{PPTSVD2} correspond to the same $W_k$ and $p=1$.
In this work, we extend these results by solving~\eqref{eppeq:epp} for $1<p<2$ and for
different choices of $W_k$.
Below we present results that give conditions for the existence and uniqueness of the
solution to~\eqref{eppeq:epp}.

\begin{lemma} \label{lem:epp}
The linear $p$-norm problem
  \be
      \argminx \| A\, x-b \|_p^p, \qquad p>1, \label{eppeq:pmin}
  \ee
has a unique minimizer if and only if $A$ has full column rank.
\end{lemma}

\begin{proof}
The function $\| x \|_p^p$ is strictly convex for $1 < p$, or equivalently,
the Hessian $H(x)$ of $\|x\|_p^p$ is positive definite for all $x$.
This implies that $\|Ax-b\|_p^p$ is strictly convex (or equivalently, its Hessian $A^* H(x)\, A$
is positive definite for all $x$) if and only if $A$ has full column rank.
It follows from strict convexity that the minimizer is unique.\footnote{We thank Martin S.
Andersen for help with this proof.}
$\qquad$
\end{proof}

\begin{theorem} \label{thm:epp}
The modified projection problem \eqref{eppeq:epp} has a unique minimizer if and only if
$\mN(AW_kW_k^T) \cap \mN(L) = \{ 0 \}$.
\end{theorem}

\begin{proof}
From \cite{elden82}, the constraint set $\mathcal{B}$ in~\eqref{eppeq:epp} can be written as
  \[
    \mathcal{B} = \{ x: x = (AW_kW_k^T)^\dagger b + Px', \ x' \mbox{ arbitrary} \},
  \]
where $\dagger$ denotes the Moore-Penrose pseudoinverse~\cite{bjorck96} and
  \[
    P = I-(AW_kW_k^T)^\dagger (AW_kW_k^T)
  \]
is the orthogonal projector onto $\mN(AW_kW_k^T)$.
Let $\tilde{b} = (AW_kW_k^T)^\dagger b$. Solving the constrained minimization~\eqref{eppeq:epp}
is equivalent to solving the following unconstrained problem
  \[
    \min_{\tilde{x}} \| LP\tilde{x} - (-L\tilde{b}) \|_p.
  \]
By Lemma~\ref{lem:epp}, the above minimization problem has a unique solution if and only if
$\mN(LP) = \{ 0 \}$.
This is true for $P = I-(AW_kW_k^T)^\dagger (AW_kW_k^T)$, the projection onto
$\mN(AW_kW_k^T)$, if and only if $\mN(AW_kW_k^T) \cap \mN(L) = \{ 0 \}$. $\qquad$
\end{proof}

\subsection{Algorithm}

It follows from the proof of Theorem~\ref{thm:epp} that we can solve the modified projection
problem \eqref{eppeq:epp} in two steps.
We first compute an approximate solution $x_k \in \mathcal{S}_k$ that contains the smooth components,
and then we compute the edge-correction component $x_0$ in the orthogonal complement $\mS_k^\perp$.
As a result,
  \[
    x = x_k + x_0 = W_k y_k + W_0 y_0 \ ,
  \]
where $y_k$ is the solution to the projected problem, and $y_0$ is the solution to an
associated $p$-norm problem.
These two solutions are computed sequentially, as shown in the EPP Algorithm~\ref{eppalg:eppalg}.

\begin{algorithm}[h]
\caption{Edge-Preserving Projection (EPP) Algorithm} 
\begin{algorithmic}[1]
\State Compute the smooth component $x_k = W_k y_k$ using the projected problem
  \be
  \label{eppeq:eppy}
    y_k = \argminy \| (AW_k)y-b \|_2.
  \ee
\vspace*{-4mm}
\State Compute the correction component $x_0 = W_0 y_0$ using the $p$-norm problem
  \be
    y_0 = \argminy \| (LW_0)y - (-LW_k y_k)  \|_p. \label{eppeq:eppz}
  \ee
\vspace*{-4mm}
\State The regularized solution is then
\vspace*{-2mm}
  \be
    x = W_k y_k + W_0 y_0 .
  \ee
\vspace*{-6mm}
\end{algorithmic}\label{eppalg:eppalg}
\end{algorithm}

\subsection{Choosing Projection Spaces}

From Lemma~\ref{lem:epp}, a sufficient condition for the uniqueness of $x$ is that
both $AW_k$ and $LW_0$ have full column rank, for then \eqref{eppeq:eppy} and \eqref{eppeq:eppz}
in the EPP Algorithm have unique solutions $y_k$ and $y_0$, correspondingly.
In principle, we can choose any subspace $\mS_k$ and its orthogonal complement $\mS_k^\perp$
with corresponding $W_k$ and $W_0$.
But in practice, however, in order to have a useful and efficient numerical implementation,
we must choose suitable basis vectors for $\mS_k$ with the following requirements:
\begin{itemize}
\item The matrix $W_k$ must separate signal from noise according to \eqref{eq:separate}.
\item The matrices $AW_k$ and $LW_0$ must have full column rank.
\item There are efficient algorithms to compute multiplications with $W_k$ and $W_0$
      and their transpose.
\end{itemize}

\subsubsection{Singular Vectors}

The MTSVD and PP-TSVD algorithms proposed in \cite{PPTSVD2,pptsvd96,mtsvd92} use the first
$k$ singular vectors as the basis vectors for $\mathcal{S}_k$.
In this case we have the following result.

\begin{theorem}\label{thm:svd}
Assume that $W_k = [v_1, v_2, \cdots,  v_k]$, where $ v_i$ are right singular vectors of $A$
corresponding to nonzero singular values.
Then the modified projection problem \eqref{eppeq:epp} has a unique solution if and only if
$\id\notin \mathcal{N}(A) = \mathrm{range}(W_0)$.
\end{theorem}

\begin{proof}
Since $W_k W_k^T$ is the orthogonal projector onto $\mathrm{range}(W_k)$ it
follows that
 \[
   \null( A W_k W_k^T) = \mathrm{range}(W_0) = \mathrm{span}\{ v_{k+1},\ldots,v_n \},
 \]
and the requirement from Theorem \ref{thm:epp} becomes
$\mathrm{range}(W_0) \cap \{ e \} \notin \{ 0 \}$, which is clearly satisfied
if $\id\notin \mathrm{range}(W_0)$. $\qquad$
\end{proof}

For blurring operators, the SVD-based subspace $\mS_k$ contains low-frequency components,
while $\mS_k^\perp$ contains relatively high-frequency components.
It is therefore very likely that the projection of $\id$ onto $\mS_k$ is not zero, and in fact
this is easy to check.

\subsubsection{Discrete Cosine Vectors}

Another suitable set of basis vectors for $\mathcal{S}_k$ are those associated with
spectral transforms such as the discrete sine or cosine transforms (DST or DCT)
and their multidimensional extensions~\cite{haje08,hanaol06}.
Recall that for 1D signals of length $m$, the orthogonal DCT matrix $C$ has elements
\[
  c_{ij} = \left\{
  \begin{array}{l l}
    \sqrt{\nicefrac{1}{m}} & \quad \text{if $i=0$}\\[1mm]
    \sqrt{\nicefrac{2}{m}}\,\cos\!\left(\frac{(2j+1)i\pi}{2m}\right) & \quad \text{if $i>0$}\\
  \end{array} \right.
  \quad \mbox{for} \quad i,j = 0,1,2,\cdots, m-1 \ .
\]
The 2-dimensional DCT matrix is the Kronecker product $C \otimes C$ of the
above matrix~\cite{vapi93}.
The DCT basis vectors, which are the \textit{rows} of the DCT matrix, have the
desired spectral properties.
Multiplications with $W_k$ and $W_0$ and their transposes are equivalent to computing
either a DCT transform or its inverse, which is done by fast algorithms similar to the FFT\@.
For this basis we have the following result.

\begin{theorem}\label{thm:dct}
Let the columns of $W_k$ be the first $k$ 2D DCT basis vectors.
Then the modified projection problem \eqref{eppeq:epp} has a unique solution if and only if
$\id\notin \mN(A)$.
\end{theorem}

\begin{proof}
From the definition of the DCT it follows that $w_1 = e / \| e \|_2$ and hence
$A W_k W_k^T e = A e$, and therefore
$\mathcal{N}(A W_k W_k^T) \cap \{ e \} = \{ 0 \} \Leftrightarrow
 Ae \neq 0 \Leftrightarrow e \notin \mathcal{N}(A)$.  $\qquad$
\end{proof}

\subsection{One-Dimensional Example}

We illustrate the use of the EPP algorithm with a one-dimensional test
problem, which uses the coefficient matrix $A$ from the \textsf{phillips}
test problem in \cite{RT} with dimension $n=64$.
The exact solution $\bar{x}$ is constructed to be piecewise constant, and
the right-hand side is $b=A\, \bar{x}$ (no noise).

\begin{figure}[h]
\begin{center}
\includegraphics[width= 0.85\textwidth]{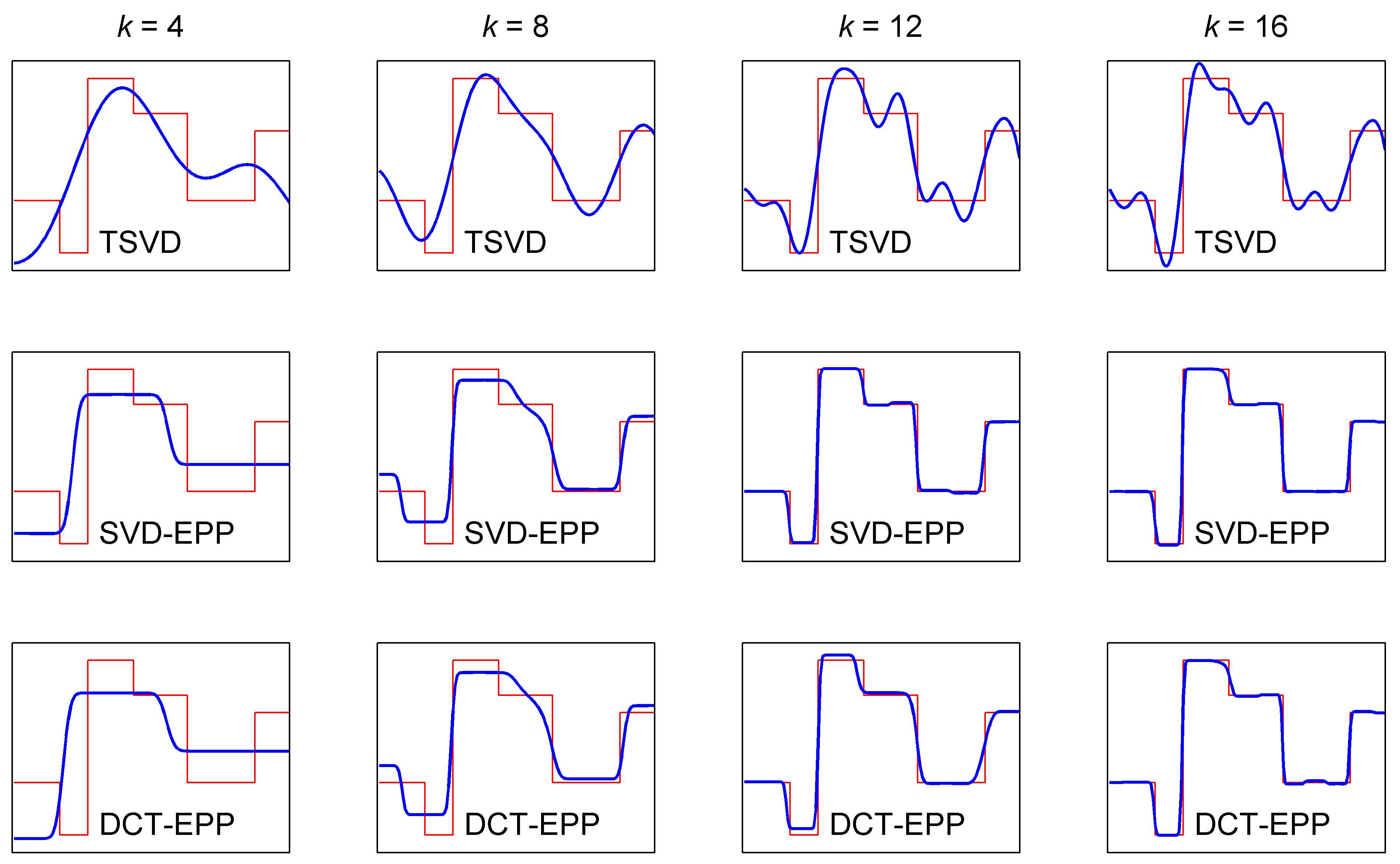}
\vspace*{-2mm}
\end{center}
\caption{\label{fig:ex1D}
  Thin red lines:\ the piecewise constant exact solution.
  Thick blue lines:\ TSVD and EPP reconstructions; $L$ is a discrete
  approximation to the first derivative operator, and $p=1.03$.}
\end{figure}

\begin{figure}[h]
\begin{center}
\includegraphics[width= 0.85\textwidth]{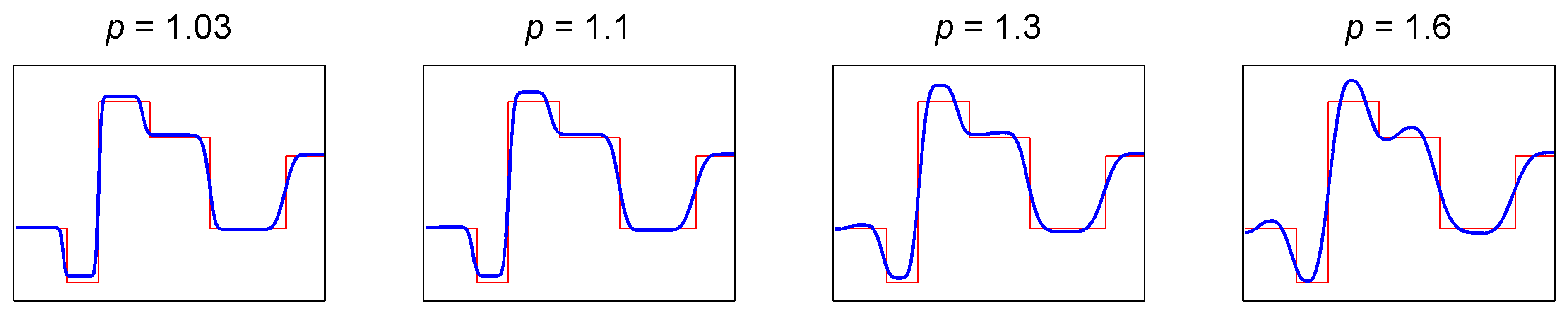}
\vspace*{-2mm}
\end{center}
\caption{\label{fig:ex1Dpp} DCT-EPP solutions for four values of $p$
  and the same $L$ as above. }
\end{figure}

Figure \ref{fig:ex1D} shows regularized solutions for four values of~$k$,
computed with the TSVD algorithm and the EPP algorithm with the SVD and
DCT bases.
The matrix $L$ is an approximation to the first derivative operator, and we use $p=1.03$.
The TSVD solutions are clearly too ``smooth'' to approximate the exact solution.
On the other hand, once $k$ is large enough that the projected component $x_k$ in
the EPP solution captures the overall structure of the solution, the EPP
algorithm is capable of producing good approximations to~$\bar{x}$
(we note that $x_k$ is identical to the TSVD solution for the SVD basis).
Figure \ref{fig:ex1Dpp} shows DCT-EPP solutions using the same $L$ as above
for four different values of~$p$, thus illustrating how $p$ controls the smoothness
of the EPP solution.

\section{Computational Issues and Numerical Implementations}
\label{eppsec:alg}

While the above analysis guarantees the existence and uniqueness of the solution to
\eqref{eppeq:epp}, it is critical to develop an efficient numerical implementation
for large-scale problems, which must take the following three issues into account:
\begin{itemize}
\item efficiently construct, or perform operations with, the basis vectors
      for the subspace~$\mS_k$,
\item robustly choose the optimal dimension $k$ of $\mS_k$, and
\item efficiently solve the $p$-norm minimization problem \eqref{eppeq:eppz}.
\end{itemize}
The optimal subspace dimension can be computed standard parameter-choice algorithms
\cite{hansenbook}; here we use the GCV method as explained in Section~\ref{eppsec:exp}.

\subsection{Working with the Projection Spaces}

As discussed above, the singular vectors and the 2D DCT matrix can be used
as the basis vectors for $\mS_k$ and $\mS_k^\perp$.
Here we will address numerical implementation issues with these choices.

For large-scale deblurring problems it is impossible to obtain $W_k = [v_1, \cdots, v_k]$
by computing the SVD of the blurring matrix $A$ without utilizing its structure.
Fortunately, in many problems the underlying point-spread function is separable or can be
approximated by a separable one \cite{hanaol06,kana99,nangpe03,vapi93}.
Hence, the blurring matrix $A$ can be represented as a Kronecker product $A_1 \otimes A_2$.
Given the SVDs of the two matrices $A_1 = U_1 \Sigma_1 V_1^T$ and $A_2 = U_2 \Sigma_2 V_2^T$,
the right singular matrix of $A$ is (or can be approximated by)
$V = \Pi (V_1 \otimes V_2)$, where the permutation matrix $\Pi$ ensures that the
ordering of the singular vectors is in accordance with decreasing singular values
(i.e., the diagonal elements of $\Pi(\Sigma_1 \otimes \Sigma_2)$).

For the DCT basis, multiplication with the $m\times m$ DCT matrix $C$ is implemented
in an very efficient way using the FFT algorithm, requiring only $O(m\log m)$ operations,
and a similar fast algorithm is available for the 2D DCT\@.
The multiplications with $W_k$ and its transpose are equivalent to applying
either the DCT or its inverse.
Therefore, it is unnecessary to form the matrix $W_k$ explicitly.

\subsection{Iteratively Reweighted Least Squares and AMG Preconditioner}

The key to the success of the EPP Algorithm is an efficient solver for the
$p$-norm minimization problem~\eqref{eppeq:eppz}.
A standard and robust approach is to use the iteratively reweighted least squares (IRLS)
method \cite{bjorck96,osborne85,irls88}, which is identical to Newton's method
with line search.
This approach reduces the $p$-norm problem to the solution of a sequence of
weighted least squares problems, which can be solved using iterative solvers.
Osborne \cite{osborne85} shows that the IRLS method is convergent for $1<p<3$.

\begin{algorithm}[t]
\caption{Iterative Reweighted Least Squares (IRLS) Algorithm
for minimizing $f(\hx)=\|\hA\,\hx -\hb\|_p$}%
\label{alg:irls}
\begin{algorithmic}[1]
\State $\hx_0 = 0$ \quad (starting vector)
\For{$j = 0, 1, 2, \ldots$ }
\State $\widehat{r}^j = \hb - \hA\,\hx^j $
\State $D_j = \diag(|r^j|^{(p-2)/2})$
\State $z^j = \mathrm{argmin}_z \| D_j ( \hA\, z - \widehat{r}^j )  \|_2$ \quad (solved iteratively)
\State $\alpha_j = \mathrm{argmin}_{\alpha} f(\hx^j+\alpha z^j)$ \quad (line search)
\State $\hx^{j+1} = \hx^j + \alpha^j z^j$
\EndFor
\end{algorithmic}
\end{algorithm}

For convenience, we briefly summarize the IRLS algorithm for solving the linear
$p$-norm problem $\min_{\hx} \| \hA \, \hx  - \hb  \|_p$.
We denote the $j$th iteration vector by $\hx^j$, and we introduce
the diagonal matrix $D_j$ determined by $j$th residual vector $\widehat{r}^j=\hb - r\hA \, \hx^j $ as
  \[
    D_j = \diag\Bigl( \bigl| \hA \, \hx^j - \hb \bigr|^{\frac{p-2}{2}} \Bigr).
  \]
The Newton search direction $z^j$ is identical to the solution of the weighted least squares problem
  \be
  \label{eppeq:irlsls}
    \min_{z} \| D_j \bigl( \hA \, z - \widehat{r}^j) \bigr \|_2.
  \ee
For $1<p<2$, as the iteration vector $\hx ^j$ gets close to the solution, the diagonal elements
in $D^j$ increase to infinity, and this tendency increases as $p$ approach~1.
Hence, the matrix $D_j\hA $ in \eqref{eppeq:irlsls} becomes increasingly ill-conditioned
as the iterations converge.
It is therefore difficult to find a suitable preconditioner for the least squares
problem~\eqref{eppeq:irlsls}.

Consider the corresponding normal equations
  \[
    \hA ^TD_j^2\hA \, z^j = \hA ^TD_j^2 \, \widehat{r}^j = \hA ^TD_j^2(\hb - \hA \, \hx^j ) ,
  \]
and define the new variable $q^{j} = z^j+\hx^j$. The normal equations can then be rewritten
  \be
  \label{eppeq:irls}
    \hA ^TD_j^2\hA \, q^{j}= \hA ^TD_j^2 \, \hb .
  \ee
The benefit of the above transformation is that the right-hand side in the new
system \eqref{eppeq:irls} depends on iteration $j$ only through $D^j$,
which is known in the $j$th iteration.\footnote{We thank Eric de Sturler for pointing this out.}

For our algorithm, it follows from \eqref{eppeq:eppz} that $\hA = LW_0$ and
$\hb = -LW_k y_k$, so \eqref{eppeq:irls} can be rewritten as
  \be
  \label{eppeq:irls2}
    W_0^T (L^TD_j^2L) W_0 \, q^j = - W_0^T (L^T D_j^2L) W_k y_k.
  \ee
Since the condition number increases as the IRLS algorithm converges to the solution,
preconditioning is necessary in solving \eqref{eppeq:irls2}.
Recall that $L$ is a gradient operator, and hence $L^TD_j^2L$ represents a diffusion
operator with large discontinuities in the diffusion coefficients.
Algebraic multi-grid (AMG) methods are robust when the diffusion coefficients are
discontinuous and vary widely~\cite{rust87, troo01}.
Therefore, we employ an AMG method to develop a right preconditioner $M$ for~\eqref{eppeq:irls2}.
The right-preconditioned problem is
  \be
  \label{eppeq:amgirls}
    [W_0^T (L^TD_j^2L) W_0 M] \tilde{q}^j = - W_0^T (L^T D_j^2L) W_k y_k,
  \ee
where $q^j = M\tilde{q}^j$.
In our implementation, given a vector $v$, the matrix-vector multiplication $w=Mv$
is implemented in three steps:
\begin{enumerate}
\item Compute $\tilde{v} = W_0 v$.
\item Use the AMG method to solve $(L^TD_j^2L)u = \tilde{v}$ for $u$.
\item Compute the result $w=W_0^Tu$.
\end{enumerate}

The matrix $W_0^T (L^TD_j^2L) W_0$ is symmetric positive definite if $D_j^2$ is positive definite.
If not, positive definiteness of $D_j^2$ can be guaranteed by adding a small positive number
to the diagonal elements.
A first thought may be to solve \eqref{eppeq:amgirls} with the conjugate gradient (CG)
method; but this requires that the preconditioner $M$ is also symmetric and positive definite.
In our implementation we use the Gauss-Seidel method in the pre- and post-relaxations of
the AMG method, and hence the AMG residual reduction operator is not symmetric~\cite{rust87},
and consequently the preconditioner is not symmetric.
Instead we solve \eqref{eppeq:irlsls} with the GMRES algorithm with right AMG
preconditioning~\cite{gmres86}.

\section{Numerical Results}
\label{eppsec:exp}

We present numerical experiments using the EPP algorithm, and we
perform a brief comparison with Total Variation deblurring.
To better visualize the impact of the high-frequency correction we use Matlab's
colormap \texttt{Hot} for the first example, which varies smoothly from black
through shades of red, orange, and yellow, to white, as the intensity increases.
Throughout we use the $256 \times 256$ ``cameraman'' test image. All the 
numerical simulations are performed using Matlab R2009b on Windows 7 x86 
32-bit system. The C compiler used to build AMG preconditioner MEX-files is 
Microsoft Visual Studio 2008. 

\subsection{Image Quality, PSFs, and Algorithm Parameters}

 \begin{figure}[t]
\centering
\includegraphics[width= 0.25\textwidth]{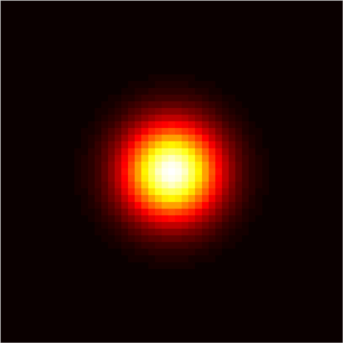}
\hspace{.25in}
\includegraphics[width= 0.25\textwidth]{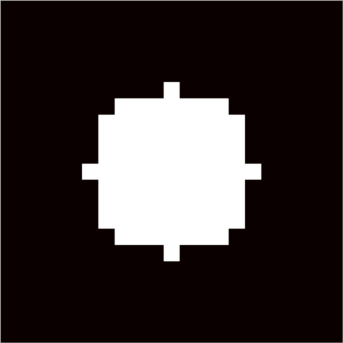}
\caption{Gaussian PSF with $\sigma=5$ (left) and out-of-focus PSF with $r=5$ (right).}
\label{fig:psfs}
\end{figure}

The ``noise level" of a test image is defined as $\| \eta \|_2/\| b \|_2$.
The quality of the restored images is measured by the relative error
$\| x_\text{restored} - \bar{x} \|_2 / \| \bar{x} \|_2$ and by the MSSIM \cite{ssim04}
(for which a larger value is better).
In our experiments the test images are generated with two common types of PSFs,
Gaussian blur and out-of-focus blur, and we use reflexive boundary conditions
in the restorations.
The elements of the \textbf{Gaussian PSF} are
  \[
    p_{ij} = \exp\!\left(-\frac{\sigma^2}{2} \Bigl( (i-k)^2 + (j-\ell)^2 \Bigr) \right),
  \]
and the elements of the \textbf{out-of-focus PSF} are
  \[
    p_{ij} = \left\{ \begin{array}{ll}
      \nicefrac{1}{\pi r^2} & \quad \text{if} \quad (i-k)^2+(j-l)^2 \leq r^2 \\[1mm]
    0 & \quad \text{otherwise} ,
  \end{array} \right.
\]
where $(k,\ell)$ is the center of the PSF, and $\sigma$ and $r$ are parameters
that determine the amount of blurring; see Fig.~\ref{fig:psfs}.
Both are doubly symmetric, but the latter is not separable, and therefore it is not
possible to efficiently compute the exact SVD of the corresponding matrix~$A$.

To compute the subspace dimension $k$ we use the GCV method \cite{hansenbook}, which can
be implemented very efficiently when the singular vectors or the DCT basis are used.
The GCV function can be expressed as
  \[
    G(k) = \frac{\sum_{i=k+1}^n \beta_i^2}{(n-k)^2}, \quad \mbox{for} \qquad k = 1, 2, \cdots, n-1,
  \]
where $\beta_i = w_i^T b$ ($w_i$ being either the left singular vectors $u_i$ or
the DCT basis vectors).
As noted in \cite{chnaol08}, the GCV method very often provides a parameter
that is too large.
Also, in some of our experiments we assume that the singular vectors are approximated by a
Kronecker product, which might be not accurate.
Hence, we choose $k$ to be equal to $2/3$ of the output from GCV algorithm,
where the factor of $2/3$ was chosen on the basis of numerous experiments.

The stopping criteria used in the iterative methods were chosen based on exhaustive experiments
(see \cite{DonghuiThesis} for details) to balance computational time against the quality of the reconstruction. 
Results computed with smaller tolerances than those used here are qualitatively similar to those computed
with the chosen tolerances, but the computational time is much longer.

\subsection{Performance of the EPP Algorithm}

In the EPP algorithm the norm parameter $p$ can be any number between $1$ and $2$.
For smaller $p$, the solution tends to have sharper edges, but as $p$ gets closer to $1$
the $p$-norm minimization in \eqref{eppeq:eppz} becomes more ill-conditioned requiring
much more computational work, while there are very little improvement of the restored images. 
Hence, we show computed results with $p=1.01, 1.05, 1.1$, and $1.2$.

 \begin{figure}[t]
\centering
\includegraphics[width= 0.24\textwidth]{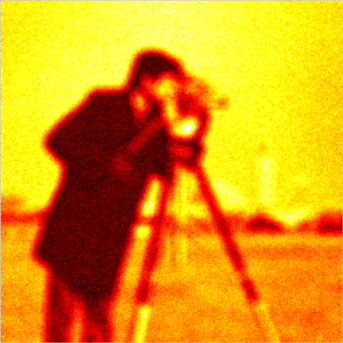}
\includegraphics[width= 0.24\textwidth]{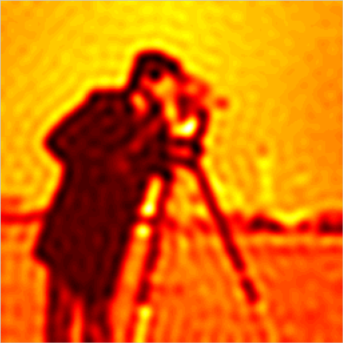}
\includegraphics[width= 0.24\textwidth]{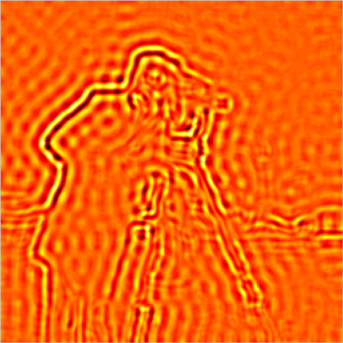}
\includegraphics[width= 0.24\textwidth]{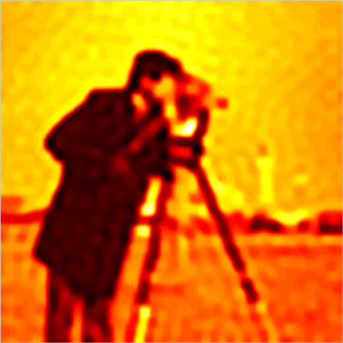}
\vspace*{-2mm}
\caption{\label{fig:dctdisk}
DCT-EPP algorithm, out-of-focus blur with $r=5$, noise level $5\%$.
Left to right:\ blurred noisy image, low-frequency component $x_k$, high-frequency
component $x_0$, and reconstruction.}
\vspace{5mm}
%
\includegraphics[width= 0.24\textwidth]{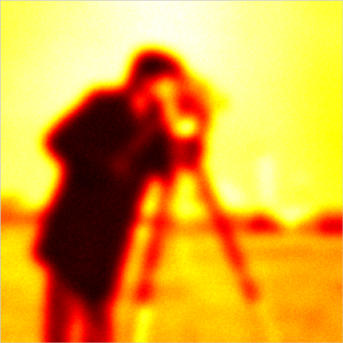}
\includegraphics[width= 0.24\textwidth]{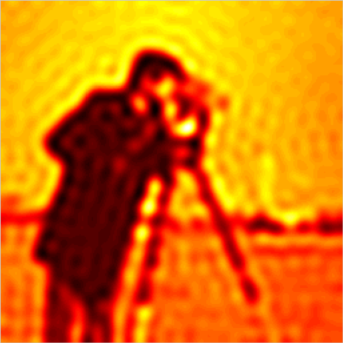}
\includegraphics[width= 0.24\textwidth]{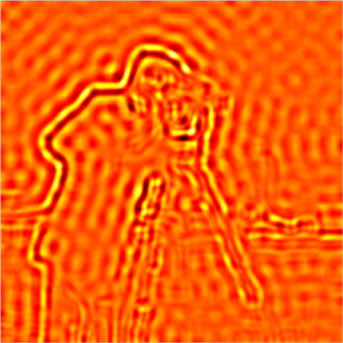}
\includegraphics[width= 0.24\textwidth]{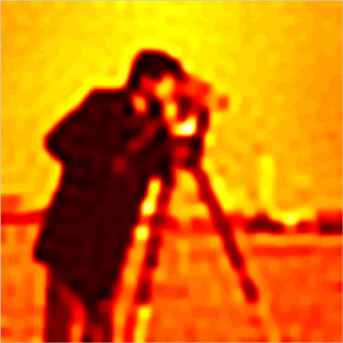}
\vspace*{-2mm}
\caption{\label{fig:dctgauss}
DCT-EPP algorithm, Gaussian blur with $\sigma=5$, noise level $1\%$.}
\vspace{5mm}
%
\centering
\includegraphics[width=0.23\textwidth]{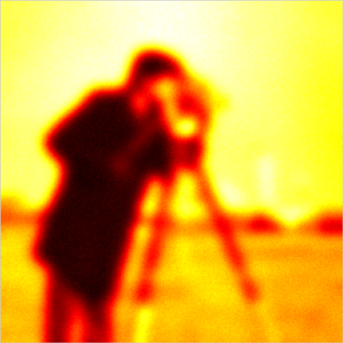}
\hspace{.01in}
\includegraphics[width=0.23\textwidth]{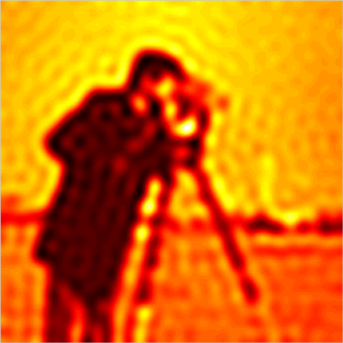}
\hspace{.01in}
\includegraphics[width=0.23\textwidth]{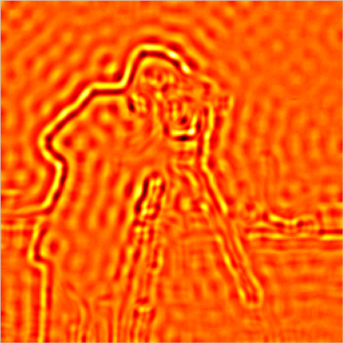}
\hspace{.01in}
\includegraphics[width=0.23\textwidth]{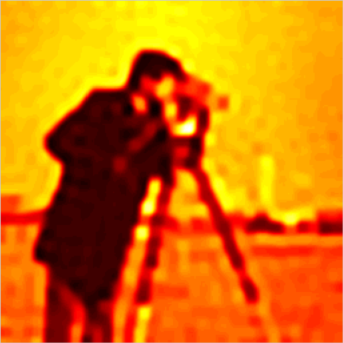}
\vspace*{-2mm}
\caption{\label{fig:svdgauss}
SVD-EPP algorithm, Gaussian blur with $\sigma=5$, noise level is $1\%$.}
\end{figure}

Table~\ref{tab:dct} shows the results of the restored out-of-focus blurred images
using the DCT-EPP algorithm.
The blur radius $r$ varies from $5$ to $15$ pixels, and the noise level
varies from 1\% to 10\%.
The table reports the computed truncation parameter $k$, the relative errors, and the MSSIM
for both $x_k$ and the final restored image.
Compared to the restored quality of $x_k$, the latter image has larger MSSIM
and smaller relative error, demonstrating that the correction step~\eqref{eppeq:eppz}
improves the image quality.
This is illustrated by the example in Figure~\ref{fig:dctdisk}.
The restored images computed using smaller $p$ are generally better than the results using larger $p$.
The corresponding results for Gaussian blur with $\sigma=5,10,15$, still using the DCT-EPP algorithm,
are also shown in Table~\ref{tab:dct}; see Fig.~\ref{fig:dctgauss} for an example.

Table \ref{tab:svd} summarizes the results for the SVD-EPP algorithm, again for out-of-focus
and Gaussian blur; see also Fig.~\ref{fig:svdgauss}.
For the Gaussian blur, the performance is similar to the DCT case.
The out-of-focus blur, however, is not separable. Therefore, we feed the SVD-EPP algorithm the
approximate singular vectors obtained from a Kronecker-product approximation of~$A$
(with Toeplitz blocks). As this particular Kronecker approximation to the true SVD is not sufficiently good,
the algorithm performs poorly, whereas use of the DCT basis gives better quality reconstructions.

\begin{landscape}
\renewcommand{\arraystretch}{1.1}
\begin{table}
\caption{Comparison of the quality of images restored by the DCT-EPP Algorithm
for $p=1.01,1.05,1.1,1.2$.
As expected, $p=1.01$ gives the best results.}~\label{tab:dct}
\centering
\begin{tabular}{|c|c|c|cc|cccc|cccc|} \hline
$r$ & noise & $k$ & \multicolumn{2}{c|}{$x_k$} & \multicolumn{8}{c|}{$x_{\mathrm{restored}}$} \\
\cline{4-13}
or & level &  & relative  & MSSIM & \multicolumn{4}{c|}{relative error} & \multicolumn{4}{c|}{MSSIM} \\
\cline{6-13}
$\sigma$ & \% &  & error &  & $p=1.01$ & $p=1.05$ & $p=1.1$ & $p=1.2$ &  $p=1.01$& $p=1.05$ &  $p=1.1$& $p=1.2$\\
\hline\hline
\multicolumn{13}{|c|}{\textbf{Out-of-focus PSF}} \\
\hline
5 & 1 & 2519 & 0.148 & 0.617& 0.135& 0.135& 0.135& 0.136& 0.707& 0.706& 0.704& 0.704  \\ 
5 & 5 & 1555 & 0.161 & 0.600& 0.151& 0.151& 0.151& 0.151& 0.668& 0.667& 0.666& 0.665  \\ 
5 & 10 & 1254 & 0.168 & 0.568& 0.159& 0.158& 0.158& 0.159& 0.637& 0.644& 0.644& 0.642  \\ 
\hline
10 & 1 & 1310 & 0.177 & 0.523& 0.158& 0.160& 0.159& 0.159& 0.640& 0.629& 0.634& 0.633  \\ 
10 & 5 & 427 & 0.204 & 0.493& 0.192& 0.192& 0.192& 0.193& 0.579& 0.578& 0.577& 0.572  \\ 
10 & 10 & 418 & 0.205 & 0.487& 0.194& 0.193& 0.194& 0.195& 0.570& 0.573& 0.565& 0.561  \\ 
\hline
15 & 1 & 772 & 0.195 & 0.488& 0.174& 0.174& 0.173& 0.175& 0.609& 0.611& 0.614& 0.604  \\ 
15 & 5 & 204 & 0.233 & 0.467& 0.223& 0.223& 0.223& 0.224& 0.531& 0.529& 0.527& 0.522  \\ 
15 & 10 & 203 & 0.234 & 0.463& 0.224& 0.224& 0.224& 0.225& 0.524& 0.523& 0.525& 0.519  \\ 
\hline\hline
\multicolumn{13}{|c|}{\textbf{Gaussian PSF}} \\
\hline
5 & 1 & 1188 & 0.168 & 0.577& 0.158& 0.158& 0.158& 0.158& 0.657& 0.656& 0.656& 0.654  \\ 
5 & 5 & 678 & 0.188 & 0.535& 0.176& 0.177& 0.177& 0.177& 0.613& 0.606& 0.610& 0.608  \\ 
5 & 10 & 560 & 0.196 & 0.508& 0.185& 0.186& 0.185& 0.186& 0.587& 0.580& 0.583& 0.579  \\ 
\hline
10 & 1 & 337 & 0.214 & 0.490& 0.202& 0.202& 0.203& 0.203& 0.563& 0.565& 0.560& 0.556  \\ 
10 & 5 & 242 & 0.228 & 0.463& 0.219& 0.219& 0.219& 0.220& 0.524& 0.524& 0.522& 0.516  \\ 
10 & 10 & 174 & 0.240 & 0.476& 0.231& 0.231& 0.231& 0.232& 0.531& 0.531& 0.530& 0.523  \\ 
\hline
15 & 1 & 167 & 0.241 & 0.471& 0.232& 0.232& 0.232& 0.233& 0.529& 0.525& 0.526& 0.522  \\ 
15 & 5 & 117 & 0.252 & 0.461& 0.240& 0.241& 0.240& 0.241& 0.525& 0.522& 0.524& 0.518  \\ 
15 & 10 & 100 & 0.257 & 0.460& 0.246& 0.246& 0.246& 0.247& 0.518& 0.517& 0.514& 0.510  \\ 
\hline
\end{tabular}
\end{table}
\end{landscape}

\begin{landscape}
\renewcommand{\arraystretch}{1.1}
\begin{table}
\caption{Comparison of the quality of images restored by the SVD-EPP Algorithm; similar to
Table~\ref{tab:dct}.}~\label{tab:svd}
\centering
\begin{tabular}{|c|c|c|cc|cccc|cccc|} \hline
$r$ & noise & $k$ & \multicolumn{2}{c|}{$x_k$} & \multicolumn{8}{c|}{$x_{\mathrm{restored}}$} \\
\cline{4-13}
or & level &  & relative  & MSSIM & \multicolumn{4}{c|}{relative error} & \multicolumn{4}{c|}{MSSIM} \\
\cline{6-13}
$\sigma$  & \% &  & error &  & $p=1.01$ & $p=1.05$ & $p=1.1$ & $p=1.2$ &  $p=1.01$& $p=1.05$ &  $p=1.1$& $p=1.2$\\
\hline\hline
\multicolumn{13}{|c|}{\textbf{Out-of-focus PSF}} \\
\hline
5 & 1 & 5648 &  0.239 &  0.448 & 0.246 & 0.246& 0.245& 0.245& 0.530 & 0.531& 0.530& 0.530  \\ 
5 & 5 & 2102 &  0.218 &  0.499 & 0.223 & 0.224& 0.223& 0.222& 0.565 & 0.564& 0.565& 0.565  \\ 
5 & 10 & 1311 &  0.219 &  0.492 & 0.219 & 0.218& 0.217& 0.216& 0.578 & 0.579& 0.578& 0.577  \\ 
\hline
10 & 1 & 2483 &  0.261 &  0.380 & 0.267 & 0.267& 0.267& 0.266& 0.443 & 0.444& 0.445& 0.444  \\ 
10 & 5 & 835 &  0.234 &  0.454 & 0.237 & 0.236& 0.236& 0.235& 0.525 & 0.525& 0.526& 0.523  \\ 
10 & 10 & 468 &  0.234 &  0.450 & 0.234 & 0.233& 0.233& 0.232& 0.535 & 0.533& 0.532& 0.527  \\ 
\hline
15 & 1 & 2173 &  0.323 &  0.202 & 0.334 & 0.335& 0.331& 0.331& 0.235 & 0.237& 0.234& 0.234  \\ 
15 & 5 & 486 &  0.254 &  0.433 & 0.253 & 0.254& 0.253& 0.252& 0.509 & 0.512& 0.509& 0.502  \\ 
15 & 10 & 336 &  0.250 &  0.421 & 0.249 & 0.250& 0.249& 0.248& 0.493 & 0.495& 0.494& 0.489  \\ 
\hline\hline
\multicolumn{13}{|c|}{\textbf{Gaussian PSF}} \\
\hline
5 & 1 & 1188 &  0.168 &  0.577 & 0.157 & 0.157& 0.157& 0.158& 0.664 & 0.662& 0.662& 0.657  \\ 
5 & 5 & 679 &  0.188 &  0.536 & 0.175 & 0.176& 0.175& 0.177& 0.621 & 0.616& 0.618& 0.611  \\ 
5 & 10 & 561 &  0.196 &  0.508 & 0.184 & 0.184& 0.185& 0.186& 0.594 & 0.589& 0.588& 0.580  \\ 
\hline
10 & 1 & 338 &  0.213 &  0.489 & 0.202 & 0.201& 0.202& 0.203& 0.564 & 0.570& 0.568& 0.559  \\ 
10 & 5 & 242 &  0.228 &  0.463 & 0.219 & 0.219& 0.219& 0.219& 0.527 & 0.523& 0.520& 0.512  \\ 
10 & 10 & 174 &  0.240 &  0.476 & 0.231 & 0.230& 0.231& 0.232& 0.535 & 0.536& 0.531& 0.524  \\ 
\hline
15 & 1 & 168 &  0.241 &  0.471 & 0.232 & 0.232& 0.232& 0.233& 0.535 & 0.533& 0.530& 0.522  \\ 
15 & 5 & 124 &  0.250 &  0.470 & 0.238 & 0.239& 0.239& 0.241& 0.534 & 0.532& 0.529& 0.519  \\ 
15 & 10 & 100 &  0.257 &  0.460 & 0.245 & 0.245& 0.245& 0.247& 0.524 & 0.520& 0.518& 0.508  \\ 
\hline
\end{tabular}
\end{table}
\end{landscape}

\subsection{Comparison with Total Variation Deblurring\label{eppsec:tv}}

We conclude by briefly comparing the performance of the EPP algorithm with the TV deblurring algorithm,
using the algorithm proposed in~\cite{hungwe08}.
In order to avoid giving our algorithm an advantage, the parameters of the TV algorithm
were chosen to optimize the MSSIM (which obviously requires the true image).
As shown in Table \ref{tab:TV}, the images restored by the TV method qualitatively
have better quality as those computed by EPP algorithm as measured by both the
relative error and the MMSIM\@.
This demonstrates that the EPP algorithm can be a computationally attractive
alternative to TV.

\begin{table}[t]
\renewcommand{\arraystretch}{1.1}
\caption{Comparison of the restored images by the TV and DCT-EPP algorithms with $p=1.01$.
There is no dramatic difference between the performance of the two algorithms.}~\label{tab:TV}
\centering
\begin{tabular}{|c|c|cc|cc|cc|cc|} \hline
$r$ & noise & \multicolumn{2}{c|}{relative error}  & \multicolumn{2}{c|}{MSSIM}
   & \multicolumn{2}{c|}{relative error}  & \multicolumn{2}{c|}{MSSIM} \\
\cline{3-10}
$\sigma$ & level \% & EPP & TV & EPP & TV & EPP & TV & EPP & TV \\
\hline\hline
 & & \multicolumn{4}{|c|}{\textbf{Out-of-focus PSF}} & \multicolumn{4}{|c|}{\textbf{Gaussian PSF}} \\
\hline
 5 & 1 & 0.135 & 0.150 & 0.707 & 0.677 & 0.158 & 0.167& 0.657 & 0.624  \\
 5 & 5 & 0.151 & 0.153 & 0.668 & 0.656 & 0.176 & 0.176& 0.613 & 0.587  \\
 5 & 10 &0.159 & 0.159 & 0.637 & 0.626 & 0.185 & 0.182& 0.587 & 0.559  \\
\hline
10 & 1 & 0.158 & 0.172 & 0.640 & 0.606 & 0.202 & 0.205& 0.563 & 0.528  \\
10 & 5 & 0.192 & 0.180 & 0.579 & 0.571 & 0.219 & 0.213& 0.524 & 0.507  \\
10 & 10 &0.194 & 0.187 & 0.570 & 0.495 & 0.231 & 0.219& 0.531 & 0.496  \\
\hline
15 & 1 & 0.174 & 0.184 & 0.609 & 0.562 & 0.232 & 0.232& 0.529 & 0.489  \\
15 & 5 & 0.223 & 0.195 & 0.531 & 0.520 & 0.240 & 0.236& 0.525 & 0.479  \\
15 & 10 &0.224 & 0.205 & 0.524 & 0.462 & 0.246 & 0.249& 0.518 & 0.455  \\
\hline
\end{tabular}
\end{table}

 \begin{figure}[h]
\centering
\includegraphics[width= 0.24\textwidth]{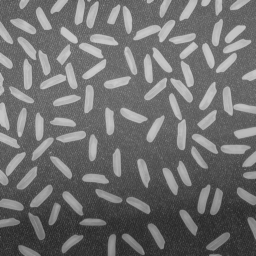}
\includegraphics[width= 0.24\textwidth]{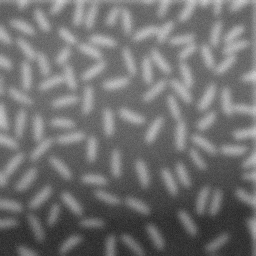}
\includegraphics[width= 0.24\textwidth]{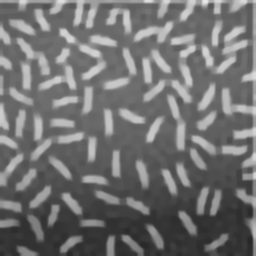}
\includegraphics[width= 0.24\textwidth]{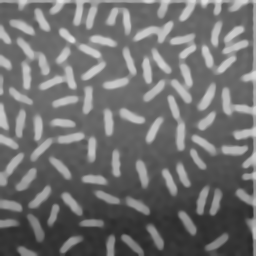}
\vspace*{-3mm}
\caption{Left to right:\ original rice grain image, blurred noisy image (Gaussian blur
with $\sigma = 3$ and noise level 0.03), DCT-EPP reconstruction
(PSNR = 24.3, MSSIM = 0.72), and TV reconstruction (PSNR = 24.9, MSSIM = 0.74).}
\label{fig:rice}
\end{figure}

To illustrate that the EPP and TV reconstructions have different features
(due to the different reconstruction models)
we consider MATLAB's $256 \times 256$
``rice grain'' image shown in Fig.~\ref{fig:rice} together with a Gaussian-blurred version
and the DCT-EPP and TV reconstructions.
The TV reconstruction has sharper edges, which comes at the expense of a more
complicated algorithm with much larger computing time.

\section{Conclusions}
\label{eppsec:con}

We developed a new computational framework for projection-based edge-preserving regularization,
and proved the existence and uniqueness of the solution.
Our algorithm uses standard computational building blocks and is therefore easy to implement
and tune to specific applications.
Our experimental results for image deblurring show that the reconstructions are better than
those from standard projection algorithms, and they are competitive with those from other
edge preserving restoration techniques.

\bibliographystyle{siam}
\bibliography{epp_ref}

\end{document}